\documentclass[a4paper]{article}
\usepackage[all]{xy}\usepackage[latin1]{inputenc}        %accents
\usepackage[dvips]{graphics,graphicx}
\usepackage{amsfonts,amssymb,amsmath,color,mathrsfs, amstext}
\usepackage{amsbsy, amsopn, amscd, amsxtra, amsthm,authblk}
\usepackage{enumerate,algorithmicx,algorithm}
\usepackage{algpseudocode}
\usepackage{upref}
\usepackage{geometry}
\usepackage[displaymath]{lineno}
\usepackage{float}
\usepackage{yhmath}
\usepackage{booktabs}
\usepackage{subcaption}
\usepackage{multirow}
\usepackage{makecell}
\usepackage[colorlinks,
            linkcolor=red,
            anchorcolor=red,
            citecolor=red
            ]{hyperref}

\numberwithin{equation}{section}

\def\e{\epsilon}

\def\bfX{\mathbf{X}}
\def\cL{\mathcal{L}}
\def\cH\mathcal{H}
\def\cA{\mathcal{A}}

\def\R{\mathbb{R}}

\def\T{\mathbb{T}}
\def\E{\mathbb{E}}

\DeclareMathOperator{\supp}{supp}

\DeclareMathOperator{\var}{var}

\newtheorem{theorem}{Theorem}[section]
\newtheorem{lemma}{Lemma}[section]

\newtheorem{remark}{Remark}[section]
\newtheorem{proposition}{Proposition}[section]
\newtheorem{assumption}{Assumption}[section]

\numberwithin{equation}{section}

\begin{document}

\title{Fluctuation suppression and enhancement in interacting particle systems}

\date{}

\author[2]{Jiaheng Chen \thanks{sjtuchenjiaheng@sjtu.edu.cn}}
\author[1]{Lei Li\thanks{leili2010@sjtu.edu.cn}}
\affil[1]{School of Mathematical Sciences, Institute of Natural Sciences and MOE-LSC, Shanghai Jiao Tong University, Shanghai, 200240, P. R. China}
\affil[2]{Zhiyuan College, Shanghai Jiao Tong University, Shanghai, 200240, P. R. China}

\maketitle

\begin{abstract}
We investigate in this work the effects of interaction on the fluctuation of empirical measures. The systems with positive definite interaction potentials tend to exhibit smaller fluctuation compared to the fluctuation in standard Monte Carlo sampling while systems with negative definite potentials tend to exhibit larger fluctuation. Moreover, if the temperature goes to zero, the fluctuation for positive definite kernels in the long time tends to vanish to zero, while the fluctuation for negative definite kernels in the long time tends to blow up to infinity. This phenomenon may gain deeper understanding to some physical systems like the Poisson-Boltzmann system, and may help to understand the properties of some particle based variational inference sampling methods.
\end{abstract}

\section{Introduction}

The interacting particle systems are ubiquitous in natural sciences  \cite{frenkel2001understanding, birdsall2004,schlick2010molecular,li2020some}, in biological sciences \cite{vicsek1995novel,bertozzi12,degond2017coagulation} and social sciences  \cite{hasimple2009,albi2013,motsch2014}. 
We consider in this work the following first order ODE/SDE systems for $N$ interacting particles
$X_i\in \bfX$ $i=1,\cdots, N$ in the mean field scaling \cite{stanley1971, georges1996, lasry2007}
\begin{gather}\label{eq:particlesystem}
dX_i= -\nabla V(X_i)dt
-\frac{1}{N}\sum_{j=1}^N \nabla W(X_i-X_j)dt
+\sqrt{2\beta^{-1}}\,dB_i,\quad i=1,\cdots, N,
\end{gather}
where $W$ is the interaction potential, $V$ is the potential for some external field and
$\{B_i\}_{i=1}^N$ are i.i.d. standard Brownian motions (or Wiener process).
$\beta^{-1}$ is the temperature of the heat bath in which these particles are placed.
We assume the state space $\bfX$ is the $d$-dimensional Euclidean space $\R^d$ or the torus $\T^d$. Here, the side length of $\T^d$ is assumed to be $2\pi$ for convenience, i.e., $\T^d=[0, 2\pi]^d$ with periodic boundary condition. 
We remark that the first order systems are rich enough in applications \cite{bertozzi12,degond2017coagulation,li2020some}
and can also be viewed as the overdamped limit (see for example \cite{freidlin2004some,hottovy2015smoluchowski}) of some second order systems (like Langevin dynamics \cite{schlick2010molecular} and the one in \cite{cucker2007emergent}).

As the number of particles $N\to\infty$, the many particle system \eqref{eq:particlesystem} can be approximated by the mean-field approximation \cite{kac1956,mckean1967,liuyang2016,jabinquantitative,li2019mean}. In particular, the  empirical measure
\begin{gather}
\mu_N:=\frac{1}{N}\sum_{i=1}^N \delta(x-X_i(t))
\end{gather}
can be shown to converge in some sense to the solution of the nonlinear Fokker-Planck equation
\begin{gather}\label{eq:nonfp}
\partial_t\mu
=\nabla\cdot((\nabla V+\nabla W*\mu)\mu)+\beta^{-1}\Delta\mu.
\end{gather}
A popular research topic is to justify this limit and the related propagation of chaos property rigorously \cite{meleard1987,sznitman1991,serfaty2020mean,jabinquantitative,li2019mean}. 

In the plasma, the interaction kernel $\Phi$ is the Coulomb interaction between electrons \cite{lazarovici2017mean,carrillo2019propagation}.
When the interaction potential is the  Coulomb potential $W=\Phi\propto \frac{1}{r^{d-2}}$ for $d\ge 3$, which is positive definite, the justification of the mean field limit is challenging. In \cite{serfaty2020mean,bresch2019mean}, the method of the modulated energy has been adopted to show the convergence, namely, almost surely
\[
 \|\mu_N-\mu\|_{\Phi}^2 \to 0, \quad N\to\infty,
\] 
where 
\[
\|f\|_{\Phi}^2=\iint_{\bfX\times\bfX} f(x)\Phi(x-y)f(y)dx dy.
\]
In the Coulomb case, $\|f\|_{\Phi}$ corresponds to the $H^{-1}$ norm of $f$.

Now that $\mu_N$ converges to $\mu$, another question goes to the fluctuation of $\mu_N$ in the large $N$ regime.
Recently,  Chen et. al.  \cite{chen2020dynamical} viewed the training process of certain two layer neural networks as the interacting particle systems. They discovered some dynamical central limit properties for training, which seems to be smaller than the direct Monte Carlo method.  Motivated by their work, we are interested in the statistical properties of $\mu_N$ in the general particle systems under the thermal equilibrium.

Let us explain the problem briefly here.  The distribution of the particles in the mean field limit (i.e., $N\to\infty$)  under the thermal equilibrium is given by the minimizer of the free energy functional (see
\cite{leble2017large,chafai2014first})
\begin{gather}\label{eq:freeenergy}
\begin{aligned}
F(\mu) &:= E(\mu)+\beta^{-1}H(\mu) \\
&=\frac{1}{2}\iint_{\bfX\times \bfX} W(x-y) \mu(dx)\mu(dy)+\int_{\bfX} V(x)\mu(dx)+\beta^{-1}H(\mu),
\end{aligned}
\end{gather}
where
\begin{gather}
E(\mu)=\frac{1}{2}\iint_{\bfX\times \bfX} W(x-y)\mu(dx)\mu(dy)+\int V(x)\mu(dx)
\end{gather}
is the energy and 
\begin{gather}
H(\mu)=
\begin{cases}\int \rho(x)\log \rho(x) dx, & \mu(dx)=\rho(x)dx,\\
+\infty, & \text{otherwise}
\end{cases}
\end{gather}
is the entropy. 
  The minimizer of \eqref{eq:freeenergy} can be shown to exist and be unique under certain conditions imposed on $V$ and the interaction potential $\Phi$ (see, for example, \cite{carrillo2003kinetic,guillin2019uniform}). The minimizer $\mu_*$ is a stationary solution of the nonlinear Fokker-Planck equation  \eqref{eq:nonfp} (see the discussion in Appendix \ref{app:F_minimizer}).
Clearly, in the mean field limit, if a particle has initial position drawn from this equilibrium $\mu_*$ and evolves according to the mean field SDE
\begin{gather}\label{eq:meanfieldsde}
dX=-\nabla V(X)\,dt-\nabla W*\mu(X)\,dt+\sqrt{2\beta^{-1}}\,dB,
\end{gather}
its distribution will be the same as $\mu_*$. Moreover, if there are $N$ particles $\bar{X}_i(0)$ drawn independently from $\mu_*$ and they evolve according to this SDE, then they will be independent from each other for any $t>0$. This means that these particles can be viewed as the Monte Carlo samplings from $\mu_*$ for every $t$.

Now that $\|\mu_N- \mu\|_{\Phi}\to 0$ almost surely and also this estimate can be uniform in time under some conditions \cite{guillin2021uniform,delarue2021uniform,rosenzweig2021global}, the $N$ particles $\{X_i\}_{i=1}^N$ can be viewed as the samples from $\mu_*$ for large $t$. 
What we are interested in is that how $\mu_N$ will compare to the direct Monte Carlo sampling. In particular, we aim to investigate whether the fluctuation in $\mu_N$ will exhibit any interesting statistical difference from the fluctuation in
\begin{gather}
\bar{\mu}_N=\frac{1}{N}\sum_{i=1}^N \delta(x-\bar{X}_i),
\end{gather}
where $\bar{X}_i(t)$'s are $N$ i.i.d. copies of \eqref{eq:meanfieldsde}.

The fluctuation in the $N\to\infty$ limit
\begin{gather}\label{eq:fluceta}
\eta :=\lim_{N\to\infty} \sqrt{N}(\mu_N-\mu)
\end{gather}
and 
\begin{gather}\label{eq:flucmean}
\bar{\eta} :=\lim_{N\to\infty} \sqrt{N}(\bar{\mu}_N-\mu)
\end{gather}
can be shown to exist, where the limit is understood in the weak sense (see, for example, \cite{wang2021gaussian}). 
Though the mean field limit points out that $\mu_N$ and $\bar{\mu}_N$ will be close if $N$ is large. We expect, however, that the fluctuations $\eta_t$ and $\bar{\eta}_t$ will be different. Motivated by the modulated energy approach in \cite{serfaty2020mean,bresch2019mean}, we assume that the interaction kernel is definite (either positive definite or negative definite, see \eqref{eq:phipd} below)
\[
W=\pm \Phi, \quad \Phi \text{ is positive definite},
\]
and compare  $ \|\eta\|_{\Phi}$
with $\|\bar{\eta}\|_{\Phi}$. A typical model with negative definite interaction kernel is the Keller-Segel model \cite{keller1970initiation, horstmann20031970}. For the system with noise, the approach in \cite{chen2020dynamical} using the flow mapping is not accessible. Instead, we make use of the SPDE that the fluctuation satisfies to perform the discussion.

The rest of the paper is organized as follows. In section \ref{sec:setup}, we introduce the basic setup of the problem and derive the main equations for the fluctuations. Using the set of eigenfunctions of the corresponding Fokker-Planck generator, we reduce the equations to a system of equations of Volterra type. In section \ref{sec:mainresult}, we present and prove our main results. Roughly speaking, we show that if the interaction potential is positive definite, the fluctuation in the interacting particle systems is suppressed, and if the interaction potential is negative definite, the fluctuation is enhanced. Moreover, when the temperature is low, the fluctuation suppression and enhancement phenomena are more obvious. We conclude the work and make a discussion in section \ref{sec:dis}.

\section{Setup and the governing equations}\label{sec:setup}

Introduce the effective potential
\begin{gather}
U(x, t)=V(x)+W*\mu=\frac{\delta E}{\delta\mu}.
\end{gather}
Under certain conditions on $V$ and $W$, the fluctuation $\eta_t$ defined in \eqref{eq:fluceta} exists and satisfies the following stochastic partial differential equation (SPDE) (see, for example, \cite{wang2021gaussian, fernandez1997hilbertian, ito1983distribution})
\begin{gather}\label{eq:eta_pde}
\partial_t\eta=\nabla\cdot(\nabla U(x, t)\eta)+\sigma \Delta\eta
+\nabla\cdot(\nabla W*\eta \mu_t)-\sqrt{2\beta^{-1}}\nabla\cdot(\sqrt{\mu_t} \xi),
\end{gather}
where $\xi$ is a space-time white noise so that
\[
\E \xi(x, t)\otimes \xi(y,  s)=I_{d\times d}\delta(x-y)\delta(s-t)
\]
 and $\mu_t$ is the solution to the mean field nonlinear Fokker-Planck equation \eqref{eq:nonfp}. This SPDE characterizes the asymptotic behavior in the $N\to\infty$ limit of the fluctuation, and should be understood in the weak sense. Moreover, it applies to systems with some singular interaction kernels like the point vortex model approximating the 2D incompressible Navier-Stokes equation and the 2D Euler equation \cite{wang2021gaussian}. We will always assume the conditions on $V$ and $W$ such that this SPDE is well-posed as in the notion in \cite{wang2021gaussian}.

% The central limit theory for interacting particle systems  \eqref{eq:particlesystem} was popularized for the Boltzmann equation in 1970-1980s

Similarly, for the fluctuation of the mean field SDE \eqref{eq:meanfieldsde}, $\bar{\eta}_t$ defined in \eqref{eq:flucmean} satisfies the following SPDE:
\begin{gather}\label{eq:bar_eta_pde}
\partial_t\bar{\eta}=\nabla\cdot(\nabla U(x, t)\bar{\eta})+\beta^{-1} \Delta\bar{\eta}-\sqrt{2\beta^{-1}}\nabla\cdot(\sqrt{\mu_t} \xi).
\end{gather}

Clearly, the properties of the fluctuation crucially rely on the transition probability (Green's function)
of the following {\it linear} Fokker-Planck equation
\begin{gather}
\partial_t p=\nabla\cdot(\nabla U(x, t) p)+\beta^{-1} \Delta p=: \cL_t^*(p).
\end{gather}
In particular, let $G(x, t, y, s)$ be the solution to 
\begin{gather}\label{eq:G}
\begin{split}
& \partial_t G(x, t, y, s)=\cL_{t,x}^* G=\nabla\cdot(\nabla U(x, t) G)+\beta^{-1} \Delta G, \quad t>s,\\
& G(x, s, y, s)=\delta(x-y).
\end{split}
\end{gather}
Here, the subindex $x$ means that the operation is acted on the argument in $G$ where $x$ occupies.

Assuming that for the two fluctuation systems, we start with the same fluctuation $\eta_0$.
In other words, we can imagine like this: we get the samples $X_i(0)=\bar{X}_i(0)$
for all $i=1,\cdots, N$ and then let them evolve according to different dynamics, one by the interacting system \eqref{eq:particlesystem} and one by the mean-field SDE \eqref{eq:meanfieldsde}.  The fluctuations will be $\eta_t$ and $\bar{\eta}_t$ respectively as $N\to\infty$, with the same initial fluctuation $\eta_0$. The fluctuations are given respectively by the Duhamel's principle by
\begin{multline}\label{eq:eta}
\eta_t=\int_{\bfX} G(x, t, y, 0)\eta_0(dy)
+\int_0^t\int_{\bfX} G(x, t, y, s)[-\sqrt{2\beta^{-1}}\nabla\cdot(\sqrt{\mu_s} \xi)](dy) ds\\
+\int_0^t \int_{\bfX} G(x, t, y, s)\nabla\cdot(\nabla W*\eta_s \mu_s)(dy) ds,
\end{multline}
and
\begin{gather}\label{eq:etabar}
\bar{\eta}_t=\int_{\bfX} G(x, t, y, 0)\eta_0(dy)
+\int_0^t\int_{\bfX} G(x, t, y, s)[-\sqrt{2\beta^{-1}}\nabla\cdot(\sqrt{\mu_s} \xi)](dy) ds.
\end{gather}

Below, we will assume that $\mu_t=\mu_*$ which is the minimizer of the functional $F(\mu)$ and the stationary solution of the nonlinear Fokker-Planck equation so that the system is in the thermal equilibrium.  In this case, $\bar{\eta}_t$ is a stationary process and for each $t$, the fluctuation is the same as the one for the Monte Carlo sampling, which is a Gaussian random field by the central limit theorem.  We first review some basic properties of the Green's function in subsection \ref{subsec:green}. Then we derive the basic equations for the fluctuations in subsection \ref{subsec:basic_eq}. Some discussion on the case without noise is made in subsection \ref{subsec:discussion_without_noise}. In subsection \ref{subsec:eigen_reduced}, we reduce the equations to a system of equations of Volterra type by using eigen-expansion.

\subsection{Basics of the linear Fokker-Planck equation}\label{subsec:green}

In this subsection, we collect some basic properties of the Fokker-Planck equation that are useful to us later.

\subsection*{The Green's function and backward equation}

Recall the definition of Green's function \eqref{eq:G}. Regarding the variables $(y, s)$, it is well-known that
\begin{gather}\label{eq:Greenback}
\partial_s G+\cL_{s, y} G=0, \quad s<t,
\end{gather}
where 
\[
\cL_{s}=-\nabla U( \cdot, s)\cdot \nabla+\beta^{-1}\Delta
\]
is the adjoint of $\cL_t^*$. The subindex $y$ means that the operation is acted on $y$ ($U$ takes value as $U(y,s)$).
 This is known as the backward Kolmogorov equation of the Green's function.  To see this relation, let us consider the solution to the equation
\[
\partial_{\tau}H(x, t; z, \tau)
+\cL_{\tau,z}^*H(x, t; z,\tau)=0,\tau<t
\quad H(x, t; z, t)=\delta(z-x).
\]
We now identify $H$ with $G$. To do this, we note the equation for $G$ so that
\[
\begin{split}
0&=\int_s^t\int_{\bfX}H(x, t; z, \tau)[\partial_{\tau}G(z, \tau; y, s)-\cL_{\tau, z} G(z, \tau; y, s)]dz d\tau\\
&=\int_{\bfX} G(z,\tau; y, s)H(x, t; z, \tau)|_{s}^t dz-\int_s^t\int_{\bfX} G(\partial_{\tau}H
+\cL_{\tau, z}^* H)dzd\tau\\
&=G(x, t; y, s)-H(x, t; y, s).
\end{split}
\]

Pick a test function $\varphi$. Define
\begin{gather}
u(y, t, s)=\E \varphi(X_t| X_s=y)=\int_{\bfX} \varphi(x)G(x, t, y, s) dx.
\end{gather}
Using the property of $G$ above, one can find that $u$ satisfies the backward Kolmogorov
equation as well
\begin{gather}
\partial_s u+\cL_{s,y} u=0.
\end{gather}

\subsection*{Time homogeneous case}

If $U(x, t)\equiv U(x)$ that is independent of $t$, then  the operators 
\[
\cL^*=\nabla\cdot(\nabla U \cdot)
+\beta^{-1}\Delta, \quad \cL=-\nabla U \cdot \nabla +\beta^{-1}\Delta
\]
are also independent of $t$, and the dynamics is time homogeneous so that
\begin{gather}
G(x, t, y, s)=G(x, t-s; y).
\end{gather}
Consequently, $\partial_t G=-\partial_s G$ and by equations \eqref{eq:G} and \eqref{eq:Greenback}, one has
\begin{gather}\label{eq:derivativeongreen}
\cL_x^*G=\cL_y G.
\end{gather}
This can also be seen formally as follows
\begin{multline*}
\cL_x^*G=\cL_x^* e^{(t-s)\cL_x^*}\delta(x-y)=
e^{(t-s)\cL_x^*}\cL_x^*\delta(x-y)\\
=e^{(t-s)\cL_x^*} \cL_y \delta(x-y)=\cL_y e^{(t-s)\cL_x^*}  \delta(x-y)
=\cL_y G.
\end{multline*}
The second last equality holds because $\cL_x^*$ and $\cL_y$ commute.
This formal verification indeed follows from the fact that $\cL_x^* G-\cL_y G=0$ for $t=s$ and that it satisfies the equation $(\partial_t-\cL_x^*)(\cL_x^* G-\cL_y G)=0$.

The relation \eqref{eq:derivativeongreen} holds only for time homogeneous case.
This relation tells that the backward Kolmogorov equation for $u(y, t)=\E(\varphi(X_t| X_0=y))$ 
now becomes
\begin{multline}\label{eq:backwardeqn}
\partial_t u=\cL_y u=\int_{\bfX} \varphi(x) \cL_y G(x, y, t-0)dx 
=\int_{\bfX} \varphi(x) \cL_x^* G(x, y, t-0)dx=\E \cL\varphi(X_t| X_0=y).
\end{multline}
The last equality is nothing but the Feynman-Kac formula, which can be derived  using It\^o's calculus directly.

Now note that when $U$ is independent of $t$,  $\cL$ is a self-adjoint nonpositive operator 
in $L^2(\mu_*)$:
\begin{gather}
\int_{\bfX} (\cL f)(y) g(y) \mu_*(dy)=\int_{\bfX}  f(y) (\cL g)(y) \mu_*(dy).
\end{gather}
In fact,
\begin{gather}
\cA:=-\cL=\nabla U(x)\cdot \nabla-\beta^{-1}\Delta =-\beta^{-1}e^{\beta U}\nabla\cdot (e^{-\beta U}\nabla).
\end{gather}
One finds that $\cA$ is a nonnegative operator easily.

Our analysis will be made based on the following assumptions on the potential $U$ and the generator
$\cL$:
\begin{assumption}\label{ass:spectralgap}
The generator $\cL=\beta^{-1}e^{\beta U}\nabla\cdot (e^{-\beta U}\nabla)$ has a spectral gap in the sense that
\[
\lambda \var(f)\le D_{\cL}(f)
\]
for some $\lambda>0$, where  $\var(f)=\mu_*( (f-\mu_*(f))^2)$ and
\[
D_{\cL}(f):=\langle-\cL  f,  f\rangle_{\mu_{*}}=\beta^{-1}\langle\nabla f,\nabla f\rangle_{\mu_*}
\]
is the corresponding Dirichlet form.
\end{assumption}

Then $\cA=-\cL$ is a nonnegative self-adjoint operator in $L^2(\mathbb{R}^d;\mu_*)$ with discrete spectrum \cite{pavliotis2014stochastic}. The eigenvalue problem for the generator is 
\[
-\cL \phi_n=\lambda_n\phi_n, \quad n=0,1,...
\]
Notice that $\phi_0=1$ and $\lambda_0=0$. The eigenvalues of the generator are real and nonnegative:
\[
0=\lambda_0<\lambda_1\le \lambda_2 \le \cdots.
\]
Furthermore, the eigenfunctions $\{\phi_j\}_{j=0}^{\infty}$ are real-valued and span $L^2(\mathbb{R}^d;\mu_*)$ in the form of a generalized Fourier series:
\[
f=\sum_{n}\phi_nf_n,\quad f_n=\langle f,\phi_n\rangle_{\mu_*}
\]
with $\langle \phi_n,\phi_m\rangle_{\mu_*}=\delta_{nm}$.

\subsection{The basic equations in the thermal equilibrium}\label{subsec:basic_eq}

Assume that the system is in the thermal equilibrium so that
$\mu_t\equiv \mu_* $, which is a minimizer of \eqref{eq:freeenergy}.  Then
\[
U(x)= V(x)+W*\mu_*(x),
\]
and $\mu_*$ satisfies
\begin{gather}
\mu_*=\rho_*(x)\,dx\propto \exp(-\beta U(x))\,dx.
\end{gather}
See Appendix \ref{app:F_minimizer} for the discussion. Here, $\rho_*(x)$ is the density of $\mu_*$.

By \eqref{eq:eta} and \eqref{eq:etabar}, one has
\begin{gather}
\eta_t=\bar{\eta}_t+\int_0^t \int_{\bfX} G(x, t, y, s)\nabla\cdot(\nabla W*\eta_s \mu_s)(dy) ds.
\end{gather}

As we mentioned, we will consider
\begin{gather}
W=\pm \Phi
\end{gather}
where $\Phi(x- y)$ is a positive semi-definite kernel. This means that for any given $x_i, 1\le i\le m$ and $c_i\in \R$, $1\le i\le m$, one has
\begin{gather}\label{eq:phipd}
\sum_{1\le i, j\le m} c_i c_j \Phi(x_i-x_j)\ge 0.
\end{gather}
In other words, the matrix $(\Phi(x_i-x_j))_{1\le i, j\le m}$ is positive semi-definite 
for any $x_i$. The Fourier transform
of $\Phi$
\begin{gather}
\hat{\Phi}(\omega)=\int_{\bfX} \Phi(x)e^{-i \omega\cdot x}dx, \quad \omega\in \hat{\bfX}
\end{gather}
is nonnegative by the Bochner's theorem \cite{rudin2017fourier}. Here, $\hat{\bfX}$ represents the space for $\omega$.
If $\bfX=\T^d$, then $\hat{\bfX}$ is discrete.
We recall that the inverse Fourier transform is given by
\begin{gather}
\Phi(x)=\frac{1}{(2\pi)^d}\int_{\hat{X}}\hat{\Phi}(\omega) e^{i\omega\cdot x}d\omega.
\end{gather}

\begin{remark}
In the case $\T^d=[0, 2\pi]^d$, 
$\Phi(x)=\frac{1}{(2\pi)^d}\sum_{\omega\in \mathbb{Z}^d}\hat{\Phi}(\omega)e^{i\omega\cdot x}$ so the coefficient in the inverse transform is also $\frac{1}{(2\pi)^d}$.   Hence, we will generally write ``$d\omega$'' as the natural measure on $\hat{\bfX}$.
If $\hat{\bfX}$ is discrete, this is the counting measure. This allows us to treat the different cases uniformly.
\end{remark}

\begin{assumption}\label{ass:kernel}
Assume that $\hat{\Phi}\ge 0$ is not identically zero, and $\hat{\Phi}\in L^1(\hat{\bfX})\cap L^2(\hat{\bfX})$ so that
\begin{gather}
\nu(d\omega):=\frac{1}{(2\pi)^d}\hat{\Phi}(\omega)d\omega
\end{gather}
is a finite measure on $\hat{\bfX}$.
\end{assumption}
Note that this measure $\nu$ corresponds to convolution with $\Phi$. Clearly,  we have
\begin{lemma}\label{eq:phiproperty}
Let Assumption \ref{ass:kernel} hold.  Then,
\begin{enumerate}
\item  $\Phi\in C_b(\bfX) \cap L^2(\bfX)$.

\item The energy norm can be written as
\begin{gather}
\|f\|_{\Phi}^2=\|\hat{f}\|_{L^2(\nu)}^2.
\end{gather}
\end{enumerate}
\end{lemma}
The first is obvious, while the second claim is a direct application of Plancherel's formula \cite{stein2011fourier}.

This motivates us to consider the evolution of the Fourier transform 
of $\eta$:
\begin{gather}
\hat{\eta}_t(\omega)=\int_{\bfX} e^{-i \omega\cdot x}\eta_t(dx),
\quad \hat{\bar{\eta}}_t(\omega)=\int_{\bfX} e^{-i \omega\cdot x}\bar{\eta}_t(dx).
\end{gather}
Both are well-defined since $e^{-i\omega\cdot x}\in C_b^{\infty}$ for any $\omega$.

\begin{proposition}
Both $\hat{\eta}_t$ and $\hat{\bar{\eta}}_t$ are Gaussian stochastic processes. Moreover, $\hat{\bar{\eta}}_t$ is a stationary process with mean zero and variance $1-|\hat{\mu}_*(\omega)|^2$ for each $t$.
They satisfy the relation 
\begin{gather}\label{eq:fluctuationrelation}
\hat{\eta}_t(\omega)=\hat{\bar{\eta}}_t(\omega)\mp
\frac{1}{(2\pi)^d}\int_0^t  \int_{\hat{\bfX}} k(\omega, \omega', t-s)  \hat{\Phi}(\omega') \hat{\eta}_s(\omega') d\omega' ds,
\end{gather}
where ``$-$'' sign corresponds to $W=\Phi$ and ``$+$'' corresponds to $W=-\Phi$ respectively, and
\begin{gather}
k(\omega, \omega', s)=\beta \int_{\bfX}  \left(e^{-\frac{1}{2}s\cA}e^{-i\omega\cdot y}\right)  \cA (e^{-\frac{1}{2}s\cA} e^{i\omega'\cdot y})  \mu_*(dy).
\end{gather}
For each $s$, $k$ is Hermitian with
\begin{gather}
k(\omega, \omega', s)=\overline{k(\omega', \omega, s)}
\end{gather}
and is positive semi-definite in $s$.
\end{proposition}

\begin{proof}
The claims about $\hat{\eta}(\omega)$ and $\hat{\bar{\eta}}(\omega)$
that they are Gaussians are clear as the equations governing $\eta$ and $\bar{\eta}$ are linear. The statistics of $\hat{\bar{\eta}}(\omega)$ follow from simple calculation.

Let 
\[
u(y, t-s; \omega)=\int_{\bfX} e^{-i\omega\cdot x}G(x, t-s; y)dx.
\]
By \eqref{eq:backwardeqn}, we find that
\begin{gather*}
u(y, t-s; \omega)= e^{(t-s)\cL}e^{-i\omega\cdot y},
\end{gather*}
which is understood as $(e^{(t-s)\cL} g)(y)$ with $g(z)=e^{-i\omega\cdot z}$.

Moreover, 
\begin{gather}
\nabla W*\eta_s=\pm \frac{1}{(2\pi)^d}\nabla_y \int_{\hat{\bfX}}\hat{\Phi}(\omega)e^{i\omega\cdot y}\hat{\eta}_s(\omega) d\omega.
\end{gather}

Since $\mu_s\equiv \mu_*$, one finds that
\begin{gather*}
\hat{\eta}_t(\omega)=\hat{\bar{\eta}}_t(\omega)\pm
\frac{1}{(2\pi)^d}\int_0^t \int_{\bfX} \left(e^{(t-s)\cL}e^{-i\omega\cdot y}\right) \int_{\hat{\bfX}}\hat{\Phi}(\omega')\nabla_y\cdot\left(\nabla_y e^{i\omega'\cdot y}\rho_*(y)\right)\hat{\eta}_s(\omega') d\omega' dy ds.
\end{gather*}
Recall that $\mu_*(dy)=\rho_*(y)dy$, $\rho_*(y)\propto \exp(-\beta U(y))$. 

Note that
\begin{gather}\label{eq:generatorrelation}
\nabla_y \cdot(\nabla_y \varphi(y)\rho_*(y))
=(\Delta_y\varphi-\beta \nabla U \cdot \nabla_y\varphi)\rho_*(y)
=\beta(\cL_y\varphi)\rho_*(y).
\end{gather}
Then, one finds
\begin{gather*}
\hat{\eta}_t(\omega)=\hat{\bar{\eta}}_t(\omega)\pm
\frac{\beta}{(2\pi)^d}\int_0^tds \int_{\hat{\bfX}} d\omega'  \hat{\Phi}(\omega') \hat{\eta}_s(\omega')  \int_{\bfX}  \left(e^{(t-s)\cL}e^{-i\omega\cdot y}\right)  (\cL_y e^{i\omega'\cdot y})  \mu_*(dy).
\end{gather*}

Introducing
\[
k(\omega, \omega', s):=-\beta\int_{\bfX}  \left(e^{s \cL}e^{-i\omega\cdot y}\right)  (\cL_y e^{i\omega'\cdot y})  \mu_*(dy).
\]

Since  $\cA=-\cL$ is self-adjoint and nonnegative in $L^2(\mu_*)$, it follows that
\[
k(\omega, \omega', s):=\beta\int_{\bfX}  \left(e^{-\frac{1}{2}s \cA}e^{-i\omega\cdot y}\right)  \cA (e^{-\frac{1}{2}s \cA} e^{i\omega'\cdot y})  \mu_*(dy).
\]
Using this formula, it is easy to find that this kernel is Hermitian for each $s$
and is positive semi-definite in time.
\end{proof}

Let us introduce the operator
\begin{gather}
Af(\omega, t)=\int_0^t\int_{\hat{\bfX}} k(\omega, \omega', t-s) f(\omega', s)\nu(d\omega') ds,
\end{gather}
which is a Volterra type operator. As a mapping on $L^2(\nu\times ds)$, it is not a symmetric operator.
The equation \ref{eq:fluctuationrelation} becomes
\begin{gather}
(I\mp A)\hat{\eta}_t=\hat{\bar{\eta}}_t.
\end{gather}

\subsection{Discussion on the case without noise}\label{subsec:discussion_without_noise}

If $\beta=+\infty$ or the temperature is zero, then the particles will have definite trajectories given the initial configuration.  For this case, one can study the evolution of the distribution using the flow maps like in \cite{chen2020dynamical}. Instead of using the flow maps, we explore briefly how the fluctuation behaves using the SPDEs for them in this subsection.

The fluctuations $\eta_t$ and $\bar{\eta}_t$ satisfy respectively the equations
\begin{gather}
\begin{split}
\partial_t\eta_t
&= \nabla\cdot(\nabla V \eta)+\nabla\cdot(\nabla W*\mu_t\eta)+\nabla\cdot(\nabla\Phi*\eta \mu_t) \\
&=\nabla\cdot(\nabla U(x, t)\eta)+\nabla\cdot(\nabla\Phi*\eta \mu_t), 
\end{split}
\end{gather}
and
\begin{gather}
 \partial_t\bar{\eta}_t= \nabla\cdot(\nabla V \bar{\eta})+\nabla\cdot(\nabla W*\mu_t\bar{\eta})
=\nabla\cdot(\nabla U(x, t)\bar{\eta}_t),
\end{gather}
where again
$U(x, t)=V+W*\mu_t$.

The Green's function for the corresponding transport equation
\begin{gather}
\partial_t p= \nabla\cdot(\nabla U(x,t) p)=\cL_{t,x}^* p
\end{gather}
is given by
\begin{gather}
G(t, x, s, y)=\delta(x-X(t; y, s)),
\end{gather}
where $X$ is the trajectory of the particle, given by
\begin{gather}
\dot{X}(t;  y, s)= -\nabla U(X, t), \quad X(s; y, s)=y.
\end{gather}
This means that the particle moves from $y$ at time $s$ to $X(t; y, s)$ at time $t$.
The mapping $y\mapsto X(t; y, s)$ is understood as the flow map. 
Clearly, 
\[
\int_{\bfX} G(t, \cdot; y, s )\mu_s(dy)=(X(t; \cdot, s))_{\#}\mu_s.
\]
This means that Green's function automatically gives the pushforward of the measure under $X$.

By the Duhamel's principle, one obtains
\begin{gather}\label{eq:etanonoise}
\begin{split}
\eta_t&=\int_{\bfX} G(t, x, 0, y)\eta_0(dy)+\int_0^t \int_{\bfX} G(t, x, s, y)\nabla\cdot(\nabla W*\eta_s \mu_s) dy ds\\
&=\bar{\eta}_t+\int_0^t \int_{\bfX} G(t, x, s, y)\nabla\cdot(\nabla W*\eta_s \mu_s) dy ds.
\end{split}
\end{gather}
The  term $\bar{\eta}_t$ is the fluctuation for the mean field dynamics and clearly for any test function
\[
\int_{\bfX} \varphi(x)\bar{\eta}_t(dx)=\int_{\bfX} \varphi(X(t; y, 0))\eta_0(dy).
\]

Taking the Fourier transform on both sides, one has
\[
\hat{\eta}_t(\omega)=\hat{\bar{\eta}}_t(\omega)
+\int_0^t \int_{\bfX}e^{-i\omega\cdot X(t; y, s)}\nabla_y\cdot(\nabla W*\eta_s \mu_s)(dy)ds.
\]
It follows that
\begin{multline}\label{eq:etaFTnn}
\hat{\eta}_t(\omega)=\hat{\bar{\eta}}_t(\omega)-\\
\frac{1}{(2\pi)^d}\int_0^t ds\int_{\hat{\bfX}}d\omega'\hat{W}(\omega') \hat{\eta}_s(\omega') \int_{\bfX} e^{-i\omega\cdot X(t; y, s)}\omega\cdot \nabla_y X(t; y,s)\cdot\omega' e^{i\omega'\cdot y} \mu_s(dy).
\end{multline}

Note that $\nabla_y X(t; y, s)$ satisfies the equation
\[
\frac{d}{dt}J_{ts}^y=-\nabla^2U(X(t; y, s))\cdot J_{ts}^y, \quad J_{ss}^y=I.
\]
Hence, $\nabla_y X(t; y,s)\cdot\omega' e^{i\omega'\cdot y}$ is nothing but $J_{ts}^y(\omega' e^{i\omega'\cdot y})$. Hence, \eqref{eq:etaFTnn} agrees with Corollary 3.3 in \cite{chen2020dynamical}.

If $\mu_*$ is a minimizer of 
\[
E(\mu)=\frac{1}{2}\int \mu(dx)W(x-y)\mu(dy)+\int V \mu(dx),
\]
then $\mu_t\equiv \mu_*$ and 
\[
\supp(\mu_*)\subset \mathrm{argmin} (W*\mu_*+V).
\]
The trajectory for $y\in \supp \mu_*$ is $X(t;y, s)\equiv y$ and $\nabla^2 U(y)$
is positive semi-definite. 
\[
J_{t,s}^y=e^{-(t-s)H(y)},\quad H(y)=\nabla^2 U(y).
\]
In this case, the kernel
\[
k(\omega, \omega', s)=\int_{\bfX} e^{-i\omega\cdot y}\omega\cdot e^{-sH(y)}\cdot\omega' e^{i\omega'\cdot y} \mu_s(dy)
\]
is Hermitian for each $s$ and positive semi-definite in time $s$. However, the structure of such case is not clear. Below, we will only focus on $\beta<+\infty$, which has more structure and might be better understood.

\subsection{Reduced system using eigen-expansion}\label{subsec:eigen_reduced}

In this subsection, we assume $\beta<+\infty$ and Assumption \ref{ass:spectralgap} holds for 
$U(x)=V(x)+W*\mu_*$. We first reduce the equation \eqref{eq:fluctuationrelation} to a Volterra integral equation taking values in $\R^{\mathbb{N}}$.

Since $e^{i\omega\cdot y}\in L^2(\mathbb{R}^d;\mu_*)$ for any $\omega$, one may write
\begin{gather}
e^{i\omega \cdot y}=\sum_{j\ge 0}c_j(\omega)\phi_j(y),\quad 
c_j(\omega)=\langle e^{i\omega \cdot y},\phi_j\rangle_{\mu_*}=\int_{\bfX} \phi_j(y)e^{i\omega\cdot y}\mu_*(dy).
\end{gather}

We have the following simple observations:
\begin{lemma}\label{lmm:cj}
\begin{itemize}
\item For any $\omega$
\begin{gather}
\sum_{j\ge 0} |c_j(\omega)|^2=\int_{\bfX} \mu_*(dy)=1.
\end{gather}

\item The family of functions $\{c_j(\omega)\}$ is linearly independent in the sense that if
\[
\sum_j \alpha_j c_j(\omega)=0, \forall \omega,
\]
then $\alpha_j=0$.
\end{itemize}
\end{lemma}
The second claim of the lemma follows from the fact that the dual basis
is given by $d_j(\omega)=\int \phi_j(y)e^{i\omega y}dy$. In other words,
\[
\int_{\hat{\bfX}} c_i(\omega) \overline{d_j(\omega)}d\omega=(2\pi)^d\delta_{ij}.
\]

It is natural to introduce the ``matrix" of size $\mathbb{N}\times\mathbb{N}$ with component
\begin{gather}
G_{ij}:=\int_{\hat{\bfX}} c_i(\omega)\bar{c}_j(\omega)\nu (d\omega).
\end{gather}
Here, we consider the operator $G$ defined by
\begin{gather}
(GX)_i=\sum_{j\ge 0} G_{ij}X_j,
\end{gather}
for a sequence $X=(X_0,X_1, \cdots)$.

\begin{proposition}\label{pro:G}
For all $i, j$
\begin{gather}\label{eq:G_ij}
G_{ij}= \iint_{\bfX\times \bfX} \Phi(y-y')\phi_i(y)\phi_j(y')\mu_*(dy)\mu_*(dy')
\in \R.
\end{gather}
The operator $G: \ell^2\to \ell^2$ is positive semi-definite. If moreover $\hat{\Phi}$ has full support in $\hat{X}$, $G$ is positive definite.
\end{proposition}

\begin{proof}
By definition, one obtains
\begin{gather*}
\begin{split}
\int_{\hat{\bfX}} c_i(\omega)\bar{c}_j(\omega)\nu (d\omega)&=\int\Big(\int e^{i\omega\cdot y}\phi_i(y)\mu_*(dy)\Big)\Big(\int e^{-i\omega \cdot y'}\phi_j(y')\mu_*(dy')\Big)\nu(d\omega)\\
&=\frac{1}{(2\pi)^d}\iint \Big(\int e^{i\omega\cdot(y-y')}\hat{\Phi}(\omega)d\omega\Big) \phi_i(y)\phi_j(y')\mu_*(dy)\mu_*(dy')\\
&= \iint_{\bfX\times \bfX} \Phi(y-y')\phi_i(y)\phi_j(y')\mu_*(dy)\mu_*(dy')\in \mathbb{R}.
\end{split}
\end{gather*}
The formula for $G_{ij}$ then follows, implying that
\begin{gather*}
    \Phi(y-y')=\sum_{i\ge 0,j\ge 0}G_{ij}\phi_i(y)\phi_j(y'),
\end{gather*}

Then, for any $X\in \ell^2$ it holds that
\[
\begin{split}
\|GX\|_{\ell^2}^2
& =\sum_i |\iint_{\bfX\times \bfX} \Phi(y-y')\phi_i(y)(\sum_j X_j \phi_j(y'))\mu_*(dy')\mu_*(dy)|^2\\
&=\left\|\int_{\bfX} \Phi(\cdot -y')(\sum_j X_j \phi_j(y'))\mu_*(dy')\right\|_{L^2(\mu_*)}^2
\end{split}
\]
Since $\Phi$ is bounded, this is clearly finite. Hence, $G: \ell^2\to \ell^2$.
Moreover, 
\begin{gather*}
\begin{aligned}
 &\langle X, GX\rangle_{\ell^2}= \sum_{i\ge 0}\sum_{j\ge 0}X_iG_{ij}X_j \\
= &\iint_{\bfX\times \bfX} \Big(\sum_{i\ge 0} X_i\phi_i(y)\Big)\Phi(y-y')\Big(\sum_{j\ge 0}X_j\phi_j(y')\Big)\mu_*(dy)\mu_*(dy')\\
= &\Big \|\sum_{i\ge 0} X_i\phi_i \mu_*\Big\|_{\Phi}^2 \ge 0,
\end{aligned}
\end{gather*}
which means that $G$ is positive semi-definite.

If $\hat{\Phi}$ has full support, then $\sum_{i\ge 0} X_i\phi_i \mu_*$ has zero Fourier transform so that (since $\mu_*$ has positive density everywhere)
\[
\sum_{i\ge 0} X_i\phi_i=0.
\]
One then must have $X_i=0$.
\end{proof}

According to Proposition \ref{pro:G}, we find that $G$ is actually the convolution operator
\[
\mathcal{G}f(x):=\int_{\bfX} \Phi(x-y)f(y)\mu_*(dy) 
\]
expressed under the basis $\{\phi_j\}_{j\ge 0}$ in $L^2(\mu_*)$. Note that in the convolution, 
the measure is $\mu_*(dy)$ instead of the Lebesgue measure.
Clearly, the square root of $G$, or $G^{1/2}$, is well-defined operator.

With the introduced family $\{c_j(\omega)\}$, one has the kernel to be
\begin{gather}\label{eq:kernelexpansion}
\begin{split}
k(\omega,\omega',s)&=\beta\int_{\bfX} \Big(\sum_{j\ge 0} e^{-\lambda_j s}\bar{c}_j(\omega)\phi_j(y)\Big)\Big(\sum_{j\ge 0}\lambda_jc_j(\omega')\phi_j(y)\Big)\mu_*(dy)\\
&=\beta\sum_{j\ge 0} \lambda_j e^{-\lambda_j s}\bar{c}_j(\omega)c_j(\omega'),
\end{split}
\end{gather}
where we recall $\lambda_0=0$.

For the convenience, we also introduce
\begin{gather}
\tilde{X}_i(t)=\int_{\bfX} \phi_i(y)\eta_t(dy),
\quad \tilde{Y}_i(t)=\int_{\bfX} \phi_i(y)\bar{\eta}_t(dy).
\end{gather}
These are the fluctuations tested on the eigenfunctions.
The vectors $\tilde{X}:=(\tilde{X}_0, \tilde{X}_1, \cdots)$
and $\tilde{Y}=(\tilde{Y}_0, \tilde{Y}_1,\cdots)$ are expected not to be in $\ell^2$. In fact, it is expected that $\E \tilde{Y}_i^2(t)=1$  for $i\ge 1$. Formally, one has
\[
\sum_i \phi_i(y)\phi_i(y')=\frac{1}{\rho_*(y)}\delta(y-y'),
\quad \sum_i \tilde{Y}_i \phi_i(y)=\frac{\bar{\eta}_t(y)}{\rho_*(y)}.
\]
Hence, one cannot hope $\tilde{Y}$ to be in $\ell^2$.
Even so, we note that 
\begin{gather}
X:=G^{1/2}\tilde{X}, \quad Y:=G^{1/2}\tilde{Y}
\end{gather}
are well-defined quantities.

The following proposition gives the basic properties of these introduced quantities.
\begin{proposition}\label{pro:basicreduced}
\begin{enumerate}[(i)]
\item  Almost surely, $X(t)=G^{1/2}\tilde{X}(t) \in \ell^2$ and 
$Y(t)=G^{1/2}\tilde{Y}(t) \in \ell^2$.

\item It holds that 
\[
\|\eta_t\|_{\Phi}^2=\|\hat{\eta}_t\|_{L^2(\nu)}^2
=\langle X, X\rangle_{\ell^2}=\langle \tilde{X}, G\tilde{X}\rangle_{\ell^2},
\]
and similar relations hold for $\bar{\eta}_t$ and $Y(t)$.

\item Introducing a family of operators $\Lambda(t): \ell^2\to \ell^2$ for $t>0$, defined by $(\Lambda(t) X)_i= \lambda_i e^{-\lambda_i t}X_i$, then the following equation holds
\begin{gather}\label{eq:XY_relation}
    X(t)=Y(t)\mp \beta\int_0^t G^{1/2}\Lambda(t-s) G^{1/2}X(s)ds,
\end{gather}
where ``$-$'' sign corresponds to $W=\Phi$ and ``$+$'' corresponds to $W=-\Phi$ respectively.
\end{enumerate}
\end{proposition}

\begin{proof}
(i+ii)    On the one hand, it is noted that one has
\begin{gather*}
\begin{aligned}
\int_{\hat{\bfX}} c_i(\omega)\hat{\eta}_t(\omega)\nu(d\omega) &=\int_{\hat{\bfX}}c_i(\omega) \left(\int e^{-i\omega y}\eta_t(dy)\right)\nu(d\omega)\\
 &=\sum_j \int_{\hat{\bfX}}c_i(\omega)\bar{c}_j(\omega)\nu(d\omega) \int_{\bfX} \phi_j(y)\eta_t(dy)\\
 &=\sum_j G_{ij}\int_{\bfX} \phi_j(y)\eta_t(dy)
 =(G\tilde{X})_i.
 \end{aligned}
\end{gather*}
Meanwhile, by direct computation,
\begin{multline*}
\int_{\hat{\bfX}} c_i(\omega)\hat{\eta}_t(\omega)\nu(d\omega)=\int_{\hat{\bfX}}\Big(\int_{\bfX} e^{i\omega\cdot y}\phi_i(y)\mu_*(dy)\Big)\hat{\eta}_t(\omega) \nu(d\omega) \\
=\frac{1}{(2\pi)^d}\int_{\bfX} \Big(\int_{\hat{\bfX}} e^{i\omega\cdot y}\widehat{\Phi*\eta_t}(\omega)d\omega\Big)\phi_i(y)\mu_*(dy)
=\langle\Phi*\eta_t, \phi_i\rangle_{\mu_*},
\end{multline*}
which implies that
\[
\Phi*\eta_t(y)=\sum_i(G\tilde{X})_i\phi_i(y).
\]

Consequently,
\begin{multline*}
  \langle X, X\rangle_{\ell^2}= \langle \tilde{X},G\tilde{X}\rangle_{\ell^2}=\sum_i(G\tilde{X})_i\int\phi_i(y)\eta_t(dy)
    =\int \Big(\sum_i (G\tilde{X})_i\phi_i(y)\Big)\eta_t(dy)\\
    =\iint_{\bfX\times\bfX} \Phi(y-y')\eta_t(dy')\eta_t(dy)
    =\|\eta_t\|_{\Phi}^2=\|\hat{\eta}_t\|_{L^2(\nu)}^2.
\end{multline*}
The last equality is due to Lemma \ref{eq:phiproperty}. 

Clearly, similar relations also hold for $\bar{\eta}_t$ and $Y(t)$
\[
\|\bar{\eta}_t\|_{\Phi}^2=\|\hat{\bar{\eta}}_t\|_{L^2(\nu)}^2
=\langle \tilde{Y}, G\tilde{Y}\rangle_{\ell^2}=\langle Y, Y\rangle_{\ell^2}.
\]

The first two claims then follow.

(iii)  Recall \eqref{eq:fluctuationrelation} and \eqref{eq:kernelexpansion}. One can derive that
\begin{gather}\label{eq:eta_bar_eta}
\begin{split}
\hat{\eta}_t(\omega)&=\hat{\bar{\eta}}_t(\omega)\mp \beta\int_0^t \sum_{j\ge 0} \lambda_j e^{-\lambda_j(t-s)}\bar{c}_j(\omega)(G\tilde{X}(s))_jds\\
&=\hat{\bar{\eta}}_t(\omega)\mp \beta\int_0^t \sum_{j\ge 0} \bar{c}_j(\omega)(\Lambda(t-s)G\tilde{X}(s))_jds.
\end{split}
\end{gather}

Now, by the definition of $c_j(\omega)$, one notes that
\[
\hat{\eta}_t(\omega)=\int_{\bfX}\sum_{j\ge 0}\bar{c}_j(\omega)\phi_j(y)\eta_t(dy)
=\sum_{j\ge 0}\bar{c}_j(\omega)X_j.
\]
Similar expression goes for $\hat{\bar{\eta}}_t$. Using the linear independence in Lemma \ref{lmm:cj}, one has
\[
\tilde{X}_j(t)=\tilde{Y}_j(t)\mp \beta\int_0^t (\Lambda(t-s) G\tilde{X}(s))_jds.
\]
Note that $\int_0^t \Lambda(t-s) G\tilde{X}(s)ds\in \ell^2$. Hence, it is well-defined to
consider acting $G^{1/2}$ on both sides to obtain
\[
X=Y\mp \beta\int_0^t G^{1/2}\Lambda(t-s) G^{1/2} X(s)ds,
\]
which holds in $\ell^2$.
\end{proof}

\section{The main results}\label{sec:mainresult}

Due to the property of $k(\omega,\omega',s)$, we expect that when $W$ is positive definite, the fluctuation in the interacting particle system is smaller, while when $W$ is negative definite, the fluctuation in the system is larger.

We first establish the statistical properties of $Y(t)$ so that it then suffices to study \eqref{eq:XY_relation} for the fluctuation.

\begin{proposition}\label{Y_stats}
The process $\tilde{Y}$ satisfies that
\[
\E \tilde{Y}_i(t) \tilde{Y}_j(s)= 
\begin{cases}
0 & i=j=0,\\
\delta_{ij}e^{-\lambda_i |t-s|} &~\text{otherwise}.
\end{cases}
\]

Consequently,
\begin{gather}
\E Y_i(t)Y_j(s)
=\Big(G^{1/2}\Big(I-\int_0^{|t-s|}\Lambda(\tau)d\tau \Big) G^{1/2}\Big)_{ij}-(G^{1/2})_{i0}(G^{1/2})_{j0}.
\end{gather}

\end{proposition}

\begin{proof}
We first note that $\tilde{Y}$ can be written as
\[
\tilde{Y}_i(t)=e^{-\lambda_i t}\int_{\bfX} \phi_i(y)\eta_0(dy)+\sqrt{2\beta^{-1}}
\int_0^t e^{-\lambda_i(t-s)}\int_{\bfX}\nabla\phi_i(y)\cdot \xi(y,s)\sqrt{\rho_*(y)}dyds.
\]
Here, $\xi$ is a vector space-time white noise such that
\[
\langle \xi(y, s)\otimes \xi(y', s')
\rangle =I \delta(y-y')\delta(s-s').
\]
Let $s\le t$. It can be verified directly that
\begin{multline*}
\E \tilde{Y}_i(t)\tilde{Y}_j(s)=e^{-\lambda_it- \lambda_js}\iint_{\bfX\times\bfX} \phi_i(y)\phi_j(y')[\mu_*(dy)\delta_y(dy')-\mu_*(dy)\mu_*(dy')]\\
+2\beta^{-1}\int_0^{s} 
e^{-\lambda_i(t-\tau)}e^{-\lambda_j(s-\tau)}
\int_{\bfX}\nabla\phi_i(y)\cdot\nabla\phi_j(y)\mu_*(dy)d\tau=: I_1+I_2.
\end{multline*}
Clearly,
\[
I_1=e^{-\lambda_i t-\lambda_j s}(\delta_{ij}-\delta_{i0}\delta_{j0}).
\]
 By \eqref{eq:generatorrelation}, one has
\begin{multline*}
I_2=-2\beta^{-1}e^{-\lambda_i(t-s)}\int_0^s e^{-\lambda_i \tau-\lambda_j \tau}d\tau\int_{\bfX}\phi_i\nabla\cdot(\nabla\phi_j(y)\rho_*(y))dy \\
=-2e^{-\lambda_i(t-s)}\int_0^s e^{-(\lambda_i+\lambda_j)\tau}d\tau\int_{\bfX}\phi_i\cL(\phi_j(y))\rho_*(y)dy.
\end{multline*}
If $j=0$, $I_2=0$. Otherwise, $I_2=\delta_{ij}(e^{-\lambda_i (t-s)}-e^{-\lambda_i t -\lambda_i s})$. Hence, 
\[
\E \tilde{Y}_i(t)\tilde{Y}_j(s)=
\begin{cases}
0 & i=j=0,\\
\delta_{ij}e^{-\lambda_i(t-s)} & \text{otherwise}.
\end{cases}
\]

Consequently,
\begin{multline*}
\E Y_i(t)Y_j(s)=\sum_{m,n\ge 0}(G^{1/2})_{im}(G^{1/2})_{jn}\E\tilde{Y}_m(t)\tilde{Y}_n(s)\\
=\sum_{n\ge1}(G^{1/2})_{in}(G^{1/2})_{jn}e^{-\lambda_n|t-s|}=\Big(G^{1/2}(I-\int_0^{|t-s|}\Lambda(\tau)d\tau) G^{1/2}\Big)_{ij}-(G^{1/2})_{i0}(G^{1/2})_{j0}.
\end{multline*}
\end{proof}

\subsection{The space homogeneous systems on torus}

In this section we set $\bfX$ to be the torus $[0, 2\pi]^d$. Consider 
\[
V(x) \equiv 0
\]
and $\mu_t(dx)\equiv\mu_*(dx)=(2\pi)^{-d}dx$. The effective potential $U(x,t)=V(x)+W*\mu \equiv \text{const.}$. The equation of fluctuation \eqref{eq:eta_pde} now becomes
\begin{gather}
\partial_t\eta=\beta^{-1}\Delta\eta
+\nabla\cdot(\nabla W*\eta \mu_t)-\sqrt{2\beta^{-1}}\nabla\cdot(\sqrt{\mu_*} \xi).
\end{gather}
Similarly, the fluctuation of the mean field limit $\bar{\eta}_t$ now satisfies
\begin{gather}\label{eq:torus_bar_eta}
\partial_t\bar{\eta}=\beta^{-1} \Delta\bar{\eta}-\sqrt{2\beta^{-1}}\nabla\cdot(\sqrt{\mu_*} \xi).
\end{gather}
Moreover, the generator in this case becomes
\begin{gather*}
    \cA=-\cL=-\beta^{-1}\Delta.
\end{gather*}

The eigenvalue problem for the Laplace operator on the torus $[0,2\pi]^d$ is 
\[
-\beta^{-1}\Delta\phi_j =\lambda_j\phi_j, \quad j=0,1,...
\]
Also, $\phi_0=1$ and $\lambda_0=0$. The eigenvalues of the generator are real and nonnegative
\[
0=\lambda_0<\lambda_1\le \lambda_2\le \cdots, \quad \lambda_j\to \infty \quad as \quad j \to \infty.
\]
In fact, the eigenvalues and the eigenfunctions are given by
\begin{gather*}
    \lambda = \beta^{-1}|k|^2,\quad \sqrt{2} \mathrm{cos}(k\cdot x), \quad   \sqrt{2} \mathrm{sin}(k\cdot x)
\end{gather*}
where $k=(k_1,k_2,\cdots,k_d)\in \mathbb{Z}^d$ and $|k|^2=k_1^2+k_2^2+\cdots+ k_d^2$. Below, we use $k(j)$ to represent the order rearrangement according to $\{\lambda_j\}_{j=0}^\infty$, namely that for $j\ge 1$
\begin{gather*}
\begin{split}
    &\lambda_{2j-1}=\beta^{-1}|k(j)|^2,\quad \phi_{2j-1}(x)=\sqrt{2} \mathrm{sin}(k(j)\cdot x);\\
    &\lambda_{2j}=\beta^{-1}|k(j)|^2, \quad \phi_{2j}(x)=\sqrt{2} \mathrm{cos}(k(j)\cdot x).
\end{split}
\end{gather*}

Furthermore, the eigenfunctions $\{\phi_j\}_{j=0}^{\infty}$ span periodic functions with period $2\pi$ in $L^2([0,2\pi]^d;\mu_*)$, which has the form of a (generalized) Fourier series:
\[
f=\sum_{j}\phi_jf_j,\quad f_j=\langle f,\phi_j\rangle_{\mu_*}
\]
with $\langle \phi_i,\phi_j\rangle_{\mu_*}=\delta_{ij}$.

Since $\Phi(x)=\Phi(-x)$, one can set the expansion of $\Phi$
\begin{gather*}
    \Phi(x)=\sum_{n\ge 0}\Phi_{2n}\phi_{2n}(x), \quad \Phi_{2n}=\langle\Phi,\phi_{2n}\rangle_{\mu_*}.
\end{gather*}
It is clear that
\[
\Phi_{2n}=\frac{\sqrt{2}}{(2\pi)^d}\hat{\Phi}(k(n))=\frac{\sqrt{2}}{(2\pi)^d}\hat{\Phi}(-k(n)), \quad n\ge 1.
\]

In this space homogeneous systems on torus, one can further simplify the problem and obtain refined results. Specifically, one has the following proposition at first
\begin{proposition}
\begin{enumerate}[(i)]
    \item G is diagonal. For $X\in \ell^2$, 
    \[
    (GX)_i=\begin{cases}  
    \Phi_0X_0  & i=0, \\
    \frac{1}{\sqrt{2}}\Phi_{2\lceil\frac{i}{2}\rceil}X_i & i \ge 1,
    \end{cases}
    \]
where $\lceil x \rceil$ is the smallest integer not less than $x$. 
    
\item It holds that for every $j$
    \begin{gather}\label{eq:X_j}
        \begin{split}
            &X_j(t)=Y_j(t)-C_j\int_0^t e^{-(C_j+\beta^{-1}|k(\lceil\frac{j}{2}\rceil)|^2)(t-s)}Y_j(s)ds, \quad if\quad W=+\Phi;\\
            &X_j(t)=Y_j(t)+C_j\int_0^t e^{-(-C_j+\beta^{-1}|k(\lceil\frac{j}{2}\rceil)|^2)(t-s)}Y_j(s)ds, \quad if\quad W=-\Phi,
        \end{split}
    \end{gather}
    where $C_j:=\frac{1}{\sqrt{2}}\Phi_{2\lceil\frac{j}{2}\rceil}|k(\lceil\frac{j}{2}\rceil)|^2$.
    
\item The process $Y$ satisfies that
\[
\E Y_i(t) Y_j(s)= 
\begin{cases}
0 & i=j=0,\\
\delta_{ij}\frac{1}{\sqrt{2}}\Phi_{2\lceil\frac{i}{2}\rceil}e^{-\lambda_i |t-s|} &~\text{otherwise}.
\end{cases}
\]
\end{enumerate}
\end{proposition}
\begin{proof}
\begin{enumerate}[(i)]
\item
Recall from \eqref{eq:G_ij} that
\begin{gather*}
\begin{split}
G_{ij}&= \iint_{\bfX\times\bfX} \Phi(y-y')\phi_i(y)\phi_j(y')\mu_*(dy)\mu_*(dy').\\
    &=\sum_{n\ge 0}\Phi_{2n} \iint_{\bfX\times\bfX} \phi_{2n}(y-y')\phi_i(y)\phi_j(y')\mu_*(dy)\mu_*(dy')\\
    &=\begin{cases}
    \delta_{ij}\Phi_0 & i=0,\\
    \frac{1}{\sqrt{2}}\delta_{ij}\Phi_{2\lceil \frac{i}{2}\rceil} &i\ge 1,
    \end{cases}
\end{split}
\end{gather*}
where the second last equality comes from the following formula
\begin{gather*}
    \phi_{2n}(y-y')=\frac{1}{\sqrt{2}}\Big(\phi_{2n}(y)\phi_{2n}(y')+\phi_{2n-1}(y)\phi_{2n-1}(y')\Big), \quad n\ge1.
\end{gather*}

\item By \eqref{eq:XY_relation}, one has
\begin{gather*}
    X(t)=Y(t)\mp \beta\int_0^t G^{1/2}\Lambda(t-s)G^{1/2}X(s) ds.
\end{gather*}
In the case that $G$ is diagonal and $\lambda_j=\beta^{-1}|k(\lceil\frac{j}{2}\rceil)|^2$, one obtains that for every $j$ 
\begin{gather}\label{eq:X_jY_j}
    X_j(t)=Y_j(t)\mp C_j\int_0^t e^{-\beta^{-1}|k(\lceil\frac{j}{2}\rceil)|^2(t-s)}X_j(s)ds,
\end{gather}
where $C_j=\frac{1}{\sqrt{2}}\Phi_{2\lceil\frac{j}{2}\rceil}|k(\lceil\frac{j}{2}\rceil)|^2$.

Then, one can solve these Volterra equations \ref{eq:X_jY_j} and find that for every $j$
\begin{gather*}
    \begin{split}
        &X_j(t)=Y_j(t)-C_j\int_0^t e^{-(C_j+\beta^{-1}|k(\lceil\frac{j}{2}\rceil)|^2)(t-s)}Y_j(s)ds, \quad if\quad W=+\Phi;\\
        &X_j(t)=Y_j(t)+C_j\int_0^t e^{-(-C_j+\beta^{-1}|k(\lceil\frac{j}{2}\rceil)|^2)(t-s)}Y_j(s)ds, \quad if\quad W=-\Phi.
      \end{split}
\end{gather*}

\item
This result is the version of Proposition \ref{Y_stats} about the space homogeneous systems on torus.
\end{enumerate}
\end{proof}

Due to the explicit expressions of $X$, we may find the pointwise estimate of the fluctuations and investigate the asymptotic behaviors.
\begin{theorem}\label{thm:torus}
Assume that $V=0$ and let Assumption \ref{ass:kernel} hold. 
\begin{enumerate}[(i)]
\item If $W=\Phi$, then $\E\|\eta_t\|_{\Phi}^2$ is decreasing in time, and for any $t>0$ 
\begin{gather*}
 \E\|\eta_t\|_{\Phi}^2 <\E\|\bar{\eta}_t\|_{\Phi}^2.
\end{gather*}
Moreover, as $t\to\infty$ one has
\[
\lim_{t\to\infty}\E \|\eta_t\|_{\Phi}^2= \sum_{j\ge 1}\frac{\E |Y_j|^2}{1+\beta \E |Y_j|^2},
\]
and consequently $\lim_{\beta\to +\infty}\lim_{t\to\infty}\|\eta_t\|_{\Phi}^2=0$.

\item

If $W=-\Phi$, then $\E\|\eta_t\|_{\Phi}^2$ is increasing in time, and for any $t>0$
\begin{gather*}
\E\|\eta_t\|_{\Phi}^2 > \E\|\bar{\eta}_t\|_{\Phi}^2, 
\end{gather*}
Moreover, there is a critical value $\beta_c$ such that when  $\beta>\beta_c$, 
$\lim_{t\to \infty}\E\|\eta_t\|_{\Phi}^2=+\infty$.
% when $\beta<\beta_c$,
% \[
% \lim_{t\to\infty}\E \|\eta_t\|_{\Phi}^2= \sum_{j\ge 1}\frac{\E |Y_j|^2}{1-\beta \E |Y_j|^2},
% \]
\end{enumerate}
\end{theorem}

\begin{proof}
\begin{enumerate}[(i)]
\item 
Recall from Proposition \ref{pro:basicreduced} that
\begin{gather*}
\|\eta_t\|_{\Phi}^2=\langle X,X\rangle_{\ell^2},\quad \|\bar{\eta}_t\|_{\Phi}^2=\langle Y,Y\rangle_{\ell^2}.
\end{gather*}

For $\|\bar{\eta}_t\|_{\Phi}^2$, it is clear that 
\[
\E\|\bar{\eta}_t\|_{\Phi}^2=\sum_{j\ge 0}\E |Y_j(t)|^2=\frac{1}{\sqrt{2}}\sum_{j\ge 1}\Phi_{2\lceil\frac{j}{2}\rceil}=\sqrt{2}\sum_{n\ge 1}\Phi_{2n}=\Phi(0)-\Phi_0,
\]
where $\Phi_0=\frac{1}{(2\pi)^d}\int_{\bfX} \Phi(y)dy \le \Phi(0)$ is due to the property of positive definite kernel.\\

Recall from \eqref{eq:X_j} that 
    \begin{gather*}
        \begin{split}
            &X_j(t)=Y_j(t)-C_j\int_0^t e^{-(C_j+\lambda_j)(t-s)}Y_j(s)ds, \quad if\quad W=+\Phi;\\
            &X_j(t)=Y_j(t)+C_j\int_0^t e^{-(-C_j+\lambda_j)(t-s)}Y_j(s)ds, \quad if\quad W=-\Phi,
        \end{split}
    \end{gather*}
where $C_j=\frac{1}{\sqrt{2}}\Phi_{2\lceil\frac{j}{2}\rceil}|k(\lceil\frac{j}{2}\rceil)|^2$ and $\lambda_j=\beta^{-1}|k(\lceil\frac{j}{2}\rceil)|^2$.\\

If $W=\Phi$, using Proposition \ref{Y_stats}, one can derive that
\[
\begin{split}
&\E|X_j(t)|^2\\
&=\E|Y_j(t)|^2-2C_j\int_0^te^{-(C_j+\lambda_j)(t-s)}\E Y_j(t)Y_j(s)ds+C_j^2\E\Big|\int_0^te^{-(C_j+\lambda_j)(t-s)}Y_j(s)ds\Big|^2\\
&=\E|Y_j|^2\Big(1+\frac{C_j}{C_j+\lambda_j}(e^{-2(C_j+\lambda_j)t}-1)\Big)\le \E|Y_j|^2,
\end{split}
\]
where we used $\E|Y_j|^2$ because this is independent of $t$.

Then,
\begin{gather*}
\E\|\eta_t\|_{\Phi}^2=\sum_{j\ge 0} \E|X_j(t)|^2\le \sum_{j\ge 0} \E|Y_j(t)|^2= \E\|\bar{\eta}_t\|_{\Phi}^2.
\end{gather*}
Moreover, it is clear that $\E\|\eta_t\|_{\Phi}^2$ is decreasing in time and
\[
\lim_{t\to\infty}\E \|\eta_t\|_{\Phi}^2= \sum_{j\ge 1}\frac{\E |Y_j|^2}{1+\beta \E |Y_j|^2}.
\]

By the dominant convergence theorem, it is found that 
\[
\lim_{\beta \to \infty}\lim_{t\to\infty}\E\|\eta_t\|_{\Phi}^2=0.
\]

\item
In the case of $W=-\Phi$, one finds that
\begin{gather*}
\E|X_j(t)|^2=
\begin{cases}
    \E|Y_j|^2\Big(1+\frac{-C_j}{-C_j+\lambda_j}(e^{-2(-C_j+\lambda_j)t}-1)\Big),  &C_j \ne \lambda_j,\\
  \E|Y_j|^2(1+2\lambda_jt), & C_j=\lambda_j,
    \end{cases}
\end{gather*}
hence $\E|X_j(t)|^2\ge \E|Y_j|^2$ and $\E|X_j(t)|^2$ is increasing in time.\\

Then it follows that
\begin{gather*}
\E\|\eta_t\|_{\Phi}^2=\sum_{j\ge 0} \E|X_j(t)|^2\ge \sum_{j\ge 0} \E|Y_j(t)|^2= \E\|\bar{\eta}_t\|_{\Phi}^2.
\end{gather*}

Note that if $\lambda_j<C_j$, namely $\beta^{-1}<\frac{1}{\sqrt{2}}\Phi_{2\lceil\frac{j}{2}\rceil}$, then $\lim_{t \to\infty}\E|X_j(t)|^2=+\infty$, which implies $\lim_{t\to\infty}\E\|\eta_t\|_{\Phi}^2=+\infty$.\\

Consequently, take 
\[
\beta_c=\min_{j\ge 1}\frac{\sqrt{2}}{\Phi_{2\lceil j/2\rceil}},
\]
if $\beta>\beta_c$, then $\lim_{t\to\infty}\E\|\eta_t\|_{\Phi}^2=+\infty$. 

% if $\beta<\beta_c$, then
% \[
% \lim_{t\to\infty}\E \|\eta_t\|_{\Phi}^2= \sum_{j\ge 1}\frac{\E |Y_j|^2}{1-\beta \E |Y_j|^2}.
% \]

\end{enumerate}
\end{proof}

In the $W=\Phi$ case, the repulsive force prevents
all the particles to collapse together.
In the $\beta\to\infty$ (zero temperature limit), the fluctuation in the interacting particle systems is expected to vanish. Theorem \ref{thm:torus} only gives a qualitative estimate regarding the vanishing rate of the fluctuation as $\beta\to \infty$. Some quantitative estimates might be obtained if one knows the decay rates of the modes in $\Phi$. For example, if 
$\Phi_{2\lceil j/2\rceil}\sim j^{-\alpha}$ with $\alpha>1$, then 
\[
\lim_{t\to\infty}\E\|\eta_t\|_{\Phi}^2\sim \beta^{-\frac{\alpha-1}{\alpha}}.
\]
This can be seen by
\[
\sum_{j}\frac{j^{-\alpha}}{1+\beta j^{-\alpha}}=\sum_{j\le \beta^{\nu}}  
+\sum_{j>\beta^{\nu}} 
\le \sum_{j\le \beta^{\nu}}\beta^{-1}
+\sum_{j>\beta^{\nu}} j^{-\alpha}
\sim \beta^{\nu-1}+\beta^{\nu(1-\alpha)}.
\]
Choosing $\nu\sim \alpha^{-1}$ gives the rate.

On the other hand, when particles attract each other, namely $W=-\Phi$, the fluctuation is increasing in time. In this case, the small temperature (smaller than a critical value) leads to the infinite fluctuation as $t\to\infty$.

The analysis for singular $\Phi$ such that $\hat{\Phi}\notin L^1$ is more challenging and left for the future. For example, the particles with Coulomb interaction in periodic box, with charge neutrality condition.

\subsection{Results in general cases}

In this part, we turn to analyze the phenomena of fluctuation suppression and enhancement in more general cases, especially for its long-time behaviour.

Recall that in the case of $W=\Phi$, our basic equation is 
\begin{gather}\label{eq:volterra1}
X(t)+\int_0^t\Gamma(t-s) X(s)ds=Y(t),
\end{gather}
where $\Gamma(t):=\beta G^{1/2}\Lambda(t)G^{1/2}$ is the Volterra kernel and of positive type  \cite[Chap. 16]{gripenberg1990volterra}. 
The Laplace transform clearly satisfies
\[
\tilde{X}(s)=(I+\tilde{\Gamma}(s))^{-1}\tilde{Y}(s).
\]
We may expect that the norm of $X$ is again smaller as in the special case discussed above. However, it is not easy to obtain the comparison for every $t$:
\[
\E\|\eta_t\|_{\Phi}^2  \le \E\|\bar{\eta}_t\|_{\Phi}^2.
\]
According to the definition of the positive type kernel, we may expect to get the comparison in the time average sense. In fact, by the positive definiteness of the kernel, one has
\begin{gather}\label{eq:generalxy}
\begin{aligned}
\int_0^T \langle X(t),Y(t)\rangle_{\ell^2}dt &=\int_0^T \langle X(t),X(t)\rangle_{\ell^2}dt
+\int_0^T \int_0^t \langle X(t), \Gamma(t-s)X(s)\rangle_{\ell^2}dsdt\\
&  \ge \int_0^T \langle X(t),X(t)\rangle_{\ell^2}dt.
\end{aligned}
\end{gather}
Hence, we may expect that the analogues of the results in Theorem \ref{thm:torus} could hold in the average sense. In fact, we have the following results.
\begin{theorem}\label{thm:general}
Suppose that Assumptions \ref{ass:spectralgap} and \ref{ass:kernel} hold. Then, we have the following claims.
\begin{enumerate}[(i)] 
    \item $(W=\Phi$, positive definite case$)$ For any $T>0$, it holds almost surely that
 \[
 \frac{1}{T}\int_0^T\|\eta_t\|_{\Phi}^2dt< \frac{1}{T}\int_0^T\|\bar{\eta}_t\|_{\Phi}^2dt.
 \]
 
\item $(W=-\Phi$, negative definite case$)$ Assume the interaction is weak such that
\[
\|G\|\le 2\beta^{-1},
\]
where $\|\cdot\|$ is the operator norm
\[
\|G\|=\mathop{\sup}_{\|X\|_{\ell^2}=1}\|GX\|_{\ell^2}=\mathop{\sup}_{\|Y\|_{\mu_*}=1}\Big\|\big\langle\Phi(y-x),Y(x)\big\rangle_{\mu_*(dx)}\Big\|_{\mu_*(dy)}.
\]
Then for any $T>0$ it holds almost surely that
\[
\frac{1}{T}\int_0^T\|\eta_t\|_{\Phi}^2dt> \frac{1}{T}\int_0^T\|\bar{\eta}_t\|_{\Phi}^2dt.
\]
\end{enumerate}
\end{theorem}

\begin{proof}
Clearly, it reduces to check $L^2(\nu)$ norms of $\hat{\eta}$ and $\hat{\bar{\eta}}$.

\begin{enumerate}[(i)]

\item 
By direct computation, 
\begin{gather*}
\begin{split}
\int_0^T\|\hat{\eta}_t\|_{L^2(\nu)}^2dt &=\int_0^T\|\hat{\bar{\eta}}_t\|_{L^2(\nu)}^2dt-2\beta\int_0^T\int_0^t\Big\langle X(t),G^{1/2}\Lambda(t-s)G^{1/2}X(s)\Big\rangle_{\ell^2}dsdt\\
&\quad\quad -\beta^2\int_0^T\Big\|\int_0^tG^{1/2}\Lambda(t-s)G^{1/2}X(s)ds\Big\|_{\ell^2}^2 dt\\
&< \int_0^T\|\hat{\bar{\eta}}_t\|_{L^2(\nu)}^2dt-2\beta\int_0^T\Big\langle X(t),\int_0^tG^{1/2}\Lambda(t-s)G^{1/2}X(s)ds\Big\rangle_{\ell^2}dt.
\end{split}
\end{gather*}
Here, it is easy to see that the term we threw away is negative.

By some direct calculations one finds that
\begin{multline*}
    \int_0^T\int_0^t \Big\langle X(t),G^{1/2}\Lambda(t-s)G^{1/2}X(s)\Big\rangle_{\ell^2}dsdt\\
    =\frac{1}{2}\sum_{j\ge 0}\lambda_je^{-2\lambda_jT}q^2_j(T)+\int_0^T\sum_{j\ge 0}\lambda_j^2e^{-2\lambda_jt}q^2_j(t)dt\ge0,
\end{multline*}
where
\[
q_j(t):=\int_0^te^{\lambda_js}(G^{1/2}X(s))_jds.
\]
It follows that
\[\int_0^T\|\hat{\eta}_t\|_{L^2(\nu)}^2dt<\int_0^T\|\hat{\bar{\eta}}_t\|_{L^2(\nu)}^2dt.\]
\item In the case of $W=-\Phi$, one has
   \begin{gather*}
   \begin{split}
    \int_0^T\|\hat{\eta}_t\|_{L^2(\nu)}^2dt &=\int_0^T\|\hat{\bar{\eta}}_t\|_{L^2(\nu)}^2dt+2\beta\int_0^T\int_0^t\Big\langle X(t),G^{1/2}\Lambda(t-s)G^{1/2}X(s)\Big\rangle_{\ell^2}dsdt\\
    &\quad\quad-\beta^2\int_0^T\Big\|\int_0^tG^{1/2}\Lambda(t-s)G^{1/2}X(s)ds\Big\|_{\ell^2}^2 dt\\
    &> \int_0^T\|\hat{\bar{\eta}}_t\|_{L^2(\nu)}^2dt+2\beta\int_0^T\int_0^t\Big\langle X(t),G^{1/2}\Lambda(t-s)G^{1/2}X(s)\Big\rangle_{\ell^2}dsdt\\
    &\quad \quad-\beta^2\|G\|\int_0^T\Big\|\int_0^t\Lambda(t-s)G^{1/2}X(s)ds\Big\|_{\ell^2}^2 dt.
   \end{split}
   \end{gather*}
    After some calculations, one finds that
    \[
    \begin{split}
    \int_0^T\Big\|\int_0^t\Lambda(t-s)G^{1/2}X(s)ds\Big\|_{\ell^2}^2 dt&=\int_0^T \sum_{j\ge 0}\lambda_j^2e^{-2\lambda_j t}\Big(\int_0^te^{\lambda_j s}(G^{1/2}X(s))_j ds\Big)^2dt\\
    &=\int_0^T\sum_{j\ge 0}\lambda_j^2e^{-2\lambda_jt}q^2_j(t)dt.
    \end{split}\]
    Then,
    \begin{multline*}
    \int_0^T\|\hat{\eta}_t\|_{L^2(\nu)}^2dt> \int_0^T\|\hat{\bar{\eta}}_t\|_{L^2(\nu)}^2dt+\beta\sum_{j\ge 0}\lambda_je^{-2\lambda_jT}q^2_j(T)\\
    +(2\beta-\beta^2\|G\|)\Big(\int_0^T\sum_{j\ge 0}\lambda_j^2e^{-2\lambda_jt}q^2_j(t)dt\Big).
    \end{multline*}
    It is clear that under the assumption of $\|G\|\le 2\beta^{-1}$, one has
    \[\int_0^T\|\hat{\eta}_t\|_{L^2(\nu)}^2dt > \int_0^T\|\hat{\bar{\eta}}_t\|_{L^2(\nu)}^2dt,\]
which concludes the proof.
\end{enumerate}
\end{proof}

Compared with the results in Theorem \ref{thm:torus}, the results in Theorem \ref{thm:general} are in almost surely sense but the average is taken along time.  
Again, we find that the fluctuation $\eta_t$ is smaller for systems with positive definite kernels  while the fluctuation could be larger for systems with negative definite kernels, compared with the fluctuation $\bar{\eta}_t$ in standard Monte Carlo sampling. 
It is more interesting to investigate the behavior 
\[
\lim_{\beta\to \infty}\lim_{T\to\infty}\frac{1}{T}\int_0^T \|\eta_t\|_{\Phi}^2.
\]
For general systems, the values of the kernel at different time points $\Gamma(t_1)$ and $\Gamma(t_2)$ do not commute, which brings difficulty in analysis. We leave the analysis of this asymptotic behavior for future study.

Below, we perform some discussion using the resolvent of the Volterra equation.
The resolvent $\Omega$ satisfies 
\[
\Omega=\Gamma-\Gamma*\Omega=\Gamma-\Omega*\Gamma.
\]
Under some appropriate assumptions, the resolvent $\Omega$ is also of positive type \cite[Chap. 16, Theorem 5.6]{gripenberg1990volterra}, and the solution of equation \eqref{eq:volterra1} is given by 
\[
X(t)=Y(t)-\int_0^t\Omega(t-s)Y(s)ds.
\]
Since $\Omega$ is of positive type, one has
\begin{multline*}
\int_0^T \langle X(t),Y(t)\rangle_{\ell^2}dt=\int_0^T \langle Y(t),Y(t)\rangle_{\ell^2} \\
-\int_0^T\int_0^t \langle Y(t),\Omega(t-s)Y(s)\rangle_{\ell^2}dsdt\le \int_0^T \langle Y(t),Y(t)\rangle_{\ell^2}.
\end{multline*}
Together with \eqref{eq:generalxy}, one has the desired comparison.

When $W=-\Phi$, the basic equation turns out to be
\begin{gather}\label{eq:volterra2}
X(t)-\int_0^t\Gamma(t-s) X(s)ds=Y(t).
\end{gather}
Similarly, if there exists some $\tilde{\Omega}$ satisfying
\begin{gather}\label{eq:tilde_Omega}
\tilde{\Omega}=\Gamma+\Gamma*\tilde{\Omega} = \Gamma+\tilde{\Omega}*\Gamma.
\end{gather}
then one can solve \eqref{eq:volterra2} as 
\[
X(t)=Y(t)+\int_0^t \tilde{\Omega}(t-s)Y(s) ds.
\]
We expect that if the interaction is weak and the $L^1$ norm of $\Gamma$ is small, then $\tilde{\Omega}$ is of positive type.

The result in Theorem \ref{thm:general} (ii) requires the type of inequality
\[
q \int_0^T\int_0^t\Big\langle X(t),\Gamma(t-s)X(s)\Big\rangle_{\ell^2}dsdt\ge \int_0^T\Big\|\int_0^t \Gamma(t-s)X(s)ds\Big\|_{\ell^2}^2 dt.
\]
This in fact requires the anti-coercivity of the kernel as discussed in  \cite[Chap. 16]{gripenberg1990volterra}. The positive constant $q$ is called coercivity constant of $\Gamma$. It is clear that if $G^{1/2}\Lambda(t-s)G^{1/2}$ is of anti-coercive type with coercivity constant $q=2\beta^{-1}$, one can recover the result in Theorem \ref{thm:general} (ii).

The resolvent can be written out explicitly using a series. For example, if $W=\Phi$ that is integrable, then one has
\begin{gather}
\Omega=\sum_{j=1}^{\infty} (-1)^{j-1}\Gamma^{* j},
\end{gather}
where $\Gamma^{* j}$ means the $j$-fold convolution. 
Then,
\[
X(t)=Y(t)+\sum_{j=1}^{\infty}(-1)^j \int_0^t \Gamma^{* j}(t-s)Y(s)\,ds.
\]
This series expansion is reminiscent of the Dyson series \cite{goldberger2004collision,tully1990molecular,lu2017path}. Intuitively, the $j$th term in the series means the contribution of the fluctuation after $j$ times of interaction.

\section{Conclusion and discussion}\label{sec:dis}
We have shown that in the interacting particle systems, if interaction potential is positive definite, the fluctuation of empirical measure is suppressed to be smaller compared with the fluctuation in Monte Carlo sampling (or in the mean field dynamics) while the systems with negative definite potentials tend to exhibit larger fluctuation. 

For the space homogeneous systems on torus, we performed the explicit comparisons including pointwise (in time) estimate and the estimate of long time behaviors. As temperature goes to zero, the long time fluctuation goes to zero for positive definite interaction while the  fluctuation goes to infinity for negative definite interaction. 
In the general systems, we have analyzed the time average of the fluctuation.

Future topics include the pointwise comparison in time and the asymptotic behavior of fluctuations for $t\to\infty$ for the general systems. More interestingly, we expect that in the zero temperature limit (as $\beta\to\infty$), the fluctuation suppression and enhancement phenomena could be put to the extreme. Besides, the analysis of fluctuation in more general situations, such as not under the thermal equilibrium, will be a much more significant and interesting problem for the future.  

It is remarked that the phenomena of fluctuation suppression and enhancement may help to gain deeper understanding to some physical systems like the Poisson-Boltzmann system \cite{guo2003vlasov}. Moreover, it may help to understand the properties of some particle based variational inference sampling methods (e.g. Kernel stein discrepancy descent in \cite{korba2021kernel}). The sampling methods based on positive definite interaction potentials may give smaller fluctuation and variance compared to standard Monte Carlo sampling, which means better sampling properties.

\section*{Acknowledgement}

This work is partially supported by the National Key R\&D Program of China, Project Number 2021YFA1002800. The work of L. Li was partially sponsored by the Strategic Priority Research Program of Chinese Academy of Sciences, Grant No. XDA25010403,  NSFC 11901389, 12031013, and  Shanghai Science and Technology Commission Grant No. 20JC144100, 21JC1403700.

\appendix

\section{The minimizer of free energy functional \eqref{eq:freeenergy}}\label{app:F_minimizer}

\begin{lemma}\label{lmm:minimizerFP}
The minimizer of the free energy functional
\begin{gather*}
\begin{aligned}
F(\mu) &= E(\mu)+\beta^{-1}H(\mu) \\
&=\frac{1}{2}\iint_{\bfX\times \bfX} W(x-y)\mu(dx)\mu(dy)+\int_{\bfX} V(x)\mu(dx)+\beta^{-1}H(\mu)
\end{aligned}
\end{gather*}
satisfies
\[
\mu_*(x)=Z^{-1}\mathrm{exp}(-\beta U(x;\mu_*))\,dx,
\]
where $U(x;\mu_*):=W*\mu_*(x)+V(x)$ and $Z^{-1}$ is the normalization constant.
Consequently, $\mu_*$ is a stationary solution of the nonlinear Fokker-Planck equation 
\eqref{eq:nonfp}.
\end{lemma}

Before we present the proof, let us derive this formally, using the standard KKT conditions \cite{boyd2004convex}, for the density $\rho_*=d\mu_*/dx$.
For the minimizer of $F(\rho)$, using the KKT for $\rho\ge 0$ and $\int \rho dx=1$, one has the Lagrangian:
\[
\cL(\rho, s, \lambda) = F(\rho)-\int s(x)\rho(x)dx+\lambda\left(\int \rho dx-1 \right).
\]
 The KKT conditions give
\[
W*\rho_* + V+\beta^{-1}(\mathrm{log}\rho_*+1)-s(x)+\lambda=0, \quad s(x)\rho_*(x)=0,\quad s\ge 0.
\]
If $\rho_*(x)$ is zero somewhere, then the left hand side would be $-\infty$.
Hence, one expects that $s\equiv 0$ and $\rho_*>0$.  The result then follows. 
Below, we justify this. 

\begin{proof}
For the energy to be finite $\mu_*$ must have a density $\rho_*$ as the entropy is finite.

We first show that $\mu_*$ is positive almost everywhere (with respect to Lebesgue measure $m$). If not, there is a set $\Omega_0$ such that 
\[
\mu_*(\Omega_0)=0,\quad m(\Omega_0)>0.
\]
One can take $\Omega_0$ such that $m(\Omega_0)<\infty$. Consider 
\[
\rho_\e=(1-\e)\rho_*+\frac{\varepsilon}{m(\Omega_0)}1_{\Omega_0}.
\]
Below, we identify $\rho_\e$ with the corresponding measure $\mu_{\e}(dx)=\rho_{\e}\,dx$.
One may check that $E(\rho_{\e})$ is Lipschitz in $\e$. However, for the entropy, it is 
\[
H(\rho_\e)=
(1-\e)\int \rho \mathrm{log}\rho dx+(1-\e)\mathrm{log}(1-\e)+\e\int_{\Omega_0}\frac{1}{m(\Omega_0)}\mathrm{log}\frac{\e}{m(\Omega_0)}dx.
\]
Since there is $\e \mathrm{log}\e$, one has $\e\log \e+C\e<0$ for any $C>0$ when $\e$ is small enough. Thus, one must have $F(\rho_{\e})<F(\rho_*)$ for some $\e$ small. This is a contradiction.

Next, we check that there exists $\gamma(\beta ; \mu_*)$ such that
\[
U(x;\mu_*)+\beta^{-1}\mathrm{log}\rho_*(x)=\gamma(\beta; \mu_*).
\]
To see this, consider the set
\[
\Gamma_{\varepsilon_0}:=\{x:\rho_*(x)\ge \varepsilon_0\}.
\]
For some $\varepsilon_0$ small enough, $m(\Gamma_{\varepsilon_0})>0$. For bounded measurable function $v$ with compact support $\mathrm{supp}(v)\subset\Gamma_0$ satisfying $\int v dx=0$, $\rho_*+\e v$ is still a probability density for $|\e|<\varepsilon_0$. Hence (by identifying $\rho_*+\e v$ with $\mu_*+\e v$)
\[
\frac{d}{d\e}F(\rho_*+\e v)|_{\e=0}=0.
\]
That is
\[
\int (U(x;\mu_*)+\beta^{-1}\mathrm{log}\rho_*(x))v dx=0.
\]
Hence, $U(x;\mu_*)+\beta^{-1}\mathrm{log}\rho_*(x)$ is a constant on $\Gamma_{\varepsilon_0}$ (even if the set is is disconnected, the constant should be the same). Since $\varepsilon_0$ is arbitrary and $\Gamma_{\varepsilon_0}$ is monotone in $\varepsilon_0$, the claim then follows. 
\end{proof}

\bibliographystyle{plain}
\bibliography{meanfield}

\end{document}